\documentclass[12pt]{article}
\usepackage{amsmath}
\usepackage{amsfonts}
\usepackage{latexsym}
\usepackage{amssymb}
\usepackage{amsthm}
\usepackage[normalem]{ulem}
\usepackage{csquotes}
\usepackage{lineno,xcolor}
\usepackage{url}
\usepackage{tikz}
\usepackage[normalem]{ulem}
\usepackage{filecontents}
\usepackage{scalefnt}
\usepackage{tikz}
\usepackage{fancyhdr}
\usetikzlibrary{arrows}
\usetikzlibrary{positioning}

\usepackage{lineno,xcolor}
\definecolor{amaranth}{rgb}{0.9, 0.17, 0.31}
\definecolor{bluegray}{rgb}{0.4, 0.6, 0.8}

\makeatletter

\usepackage{longtable}
\usepackage{multirow}
\usepackage{hyperref}

\newtheorem*{maintheorem*}{Main Theorem}
\newtheorem{theorem}{Theorem}[section]
\newtheorem{proposition}[theorem]{Proposition}
\newtheorem{corollary}[theorem]{Corollary}
\newtheorem{lemma}[theorem]{Lemma}
\newtheorem{definition}{Definition}
\newtheorem*{theorem*}{Theorem}
\newtheorem{remark}[theorem]{Remark}

\newtheorem*{example*}{Example}
\newtheorem*{conjecture*}{Conjecture}
\def\1{\mathbf 1}

\def\0{\mathbf 0}

\def\cP{\mathcal P}
\def\cCP{\mathcal{CP}}
\def\cX{\mathcal X}

\def\cT{\mathcal T}

\def\mod{{\rm mod} }

\def\<{\langle}
\def\>{\rangle}
\newcommand\comment[1]{}

\newcommand*{\shifttext}[2]{
	\settowidth{\@tempdima}{#2}
	\makebox[\@tempdima]{\hspace*{#1}#2}
}

\newcommand\redsout{\bgroup\markoverwith{\textcolor{amaranth}{\rule[0.5ex]{2pt}{0.4pt}}}\ULon}

\makeatletter
\def\@fnsymbol#1{\ensuremath{\ifcase#1\or *\or \dagger\or \ddagger\or
		\mathsection\or \mathparagraph\or \|\or **\or \dagger\dagger
		\or \ddagger\ddagger \else\@ctrerr\fi}}
\makeatother

\title{The Perfect Matching Hamiltonian property in Prism and Crossed Prism graphs}

\author{Francesco Colangelo\thanks{Francesco Colangelo: francesco.colangelo@unibas.it\hfill\newline \hspace*{1.4em}
		Universit\`{a} degli Studi della Basilicata, Dipartimento di Scienze di Base e Applicate, Viale dell'Ateneo Lucano 10, 85100 Potenza, Italy}
	\and
	Federico Romaniello\thanks{Federico Romaniello: federico.romaniello@unibas.it\hfill\newline \hspace*{1.4em}
		Universit\`{a} degli Studi della Basilicata, Dipartimento di Innovazione Umanistica, Scientifica e Sociale, Viale dell'Ateneo Lucano 10, 85100 Potenza, Italy}}
\date{}

\begin{document}
\maketitle
\begin{abstract}
A graph $G$ has the \emph{Perfect Matching Hamiltonian property} (or for short, $G$ is $PMH$) if, for each one of its perfect matchings, there is another perfect matching of $G$ such that the union of the two perfect matchings yields a Hamiltonian cycle of $G$.
In this note, we show that \emph{Prism graphs} $\cP_n$ are not $PMH$, except for the $Cube\ graph$, and indicate for which values of $n$ the \emph{Crossed Prism graphs} $\cCP_n$ are $PMH$.\\

\textbf{Math. Subj. Class.:} 05C15, 05C45, 05C70.

\end{abstract}

\section{Introduction}


Let \( G \) be a simple connected graph of even order, with vertex set \( V(G) \) and edge set \( E(G) \). A \emph{\( k \)-factor} of \( G \) is defined as a \( k \)-regular spanning subgraph of \( G \) (not necessarily connected). Two prominent topics in graph theory are \emph{perfect matchings} and \emph{Hamiltonian cycles}. A perfect matching corresponds to the edge set of a 1-factor, while a Hamiltonian cycle represents a connected 2-factor of a graph. For \( t \geq 3 \), a \emph{cycle} of length \( t \) (or a \( t \)-cycle), denoted as \( C_{t} = (v_{1}, \ldots, v_{t}) \), consists of a sequence of mutually distinct vertices \( v_1, v_2, \ldots, v_{t} \) with an associated edge set \(\{v_{1}v_{2}, \ldots, v_{t-1}v_{t}, v_{t}v_{1}\}\). We refer the reader to \cite{Diestel} for any definitions not explicitly defined here.

A graph $G$ admitting a perfect matching is said to have the \emph{Perfect-Matching-Hamiltonian property} (for short the \emph{PMH-property}) if for every perfect matching $M$ of $G$ there exists another perfect matching $N$ of $G$, with $M \cap N = \emptyset$, such that the edges of $M\cup N$ induce a Hamiltonian cycle of $G$. For simplicity, a graph admitting this property is said to be $PMH$. This property was first studied in the 1970s by Las Vergnas \cite{LasVergnas} and H\"aggkvist \cite{Haggkvist}, and for more recent results about the $PMH$-property we suggest the reader to \cite{pmhlinegraphs,papillon,ThomassenEtAl,cpcq,accordions,betwixt}. If a cubic graph $G$ is $PMH$, then every perfect matching of $G$ corresponds to one of the colours of a (proper) $3$-edge-colouring of the graph, and we say that every perfect matching can be extended to a $3$-edge-colouring. This is achieved by alternately colouring the edges of the Hamiltonian cycle containing a predetermined perfect matching using two colours, and then colouring the edges not belonging to the Hamiltonian cycle using a third colour.
However, there are cubic graphs which are not $PMH$ but have every one of their perfect matchings that can be extended to a $3$-edge-colouring. The following proposition, stated in \cite{papillon} characterises all cubic graphs for which every one of their perfect matchings can be extended to a 3-edge-colouring of the graph.

\begin{proposition}[\cite{papillon}]\label{prop e2f}
Let $G$ be a cubic graph admitting a perfect matching. Every perfect matching of $G$ can be extended to a $3$-edge-colouring of $G$ if and only if all $2$-factors of $G$ contain only even cycles.
\end{proposition}

The following corollary follows immediately:
\begin{corollary}\label{cor}
Let $G$ be a cubic graph admitting a perfect matching. If there exists a $2$--factor of $G$ that contains an odd cycle, then $G$ is not PMH. 
\end{corollary}

We will now give informal definitions of both Prism graphs and Crossed Prism graphs. Formal definitions with appropriate labelling on vertices and edges for both families will be given in Section \ref{Pn} and Section \ref{CPn}, respectively.\\
A \emph{prism graph} $\cP_n$ is obtained by connecting two identical cycles with edges between corresponding vertices.  The resulting graph resembles a three-dimensional prism, hence the name. A \emph{crossed prism graph} $\cCP_n$ is a variation of the prism graph. Still, it consists of two cycles of equal length with edges between corresponding edges, and in the \textit{outer cycle} edges connect consecutive vertices. In contrast, the \textit{inner cycle} is twisted, meaning that each vertex of the inner cycle connects not directly to the corresponding vertex on the outer cycle, but instead to another vertex in a pattern that creates a crossing structure within the inner cycle. Figure \ref{prismandcprism} shows on the left the Prism graph and on the right the Crossed Prism graph, both on 16 vertices.\\

\begin{figure}[h!]
\centering
\includegraphics[width=1\textwidth]{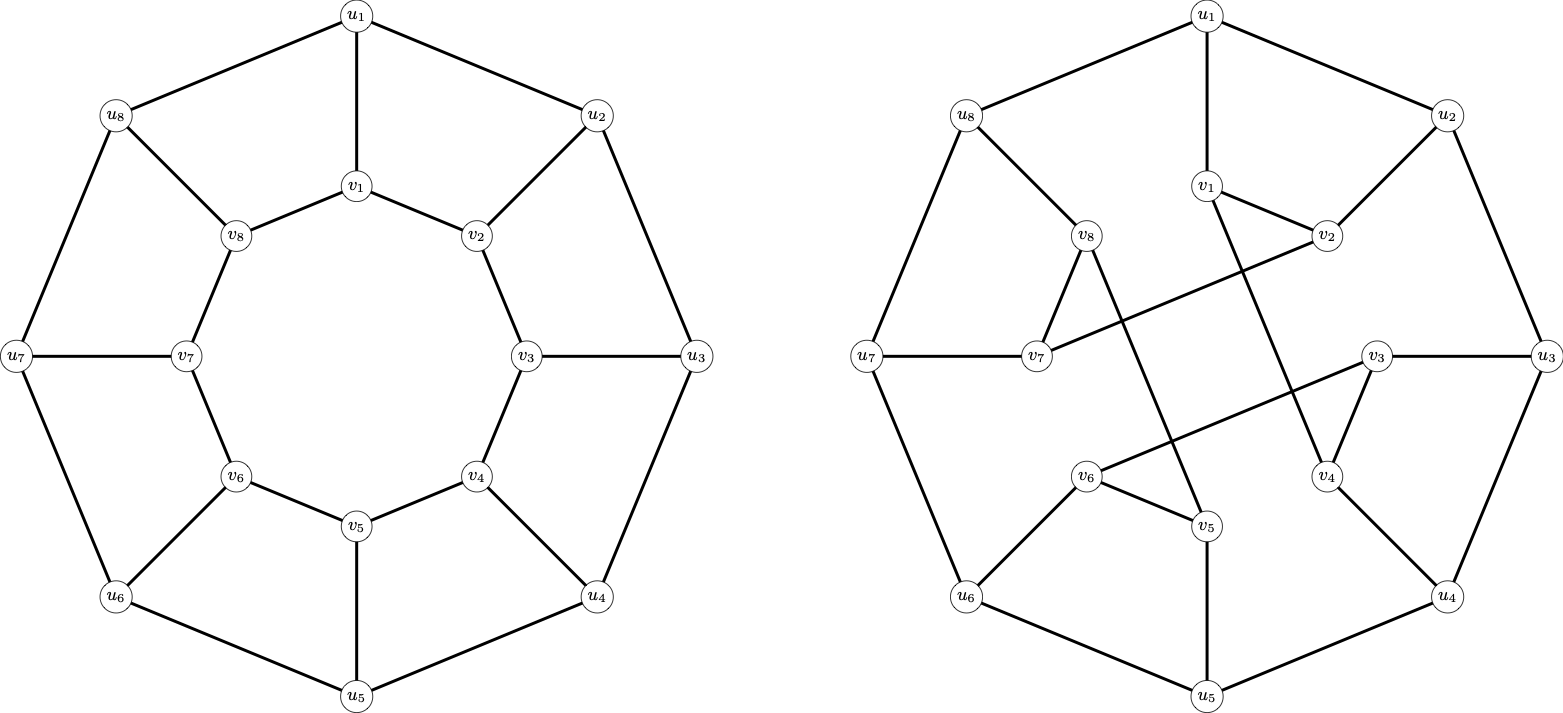}
\caption{The $8$-prism graph $\cP_8$ and the $4$-crossed prism graph $\cCP_2$, both on 16 vertices.}
\label{prismandcprism}
\end{figure}

In Section 1, we show that the only $PMH$ \emph{n-prism graph} $\cP_n$ is the Cube graph, i.e. $\cP_4$. In Section 2, we prove that the \emph{$2$n-crossed prism graph} $\cCP_n$ is $PMH$ only when $n$ is even.

\section{The PMH-property in Prism graphs}\label{Pn}
As stated in the introduction, this section shall be devoted to studying the $PMH$-property in Prism graphs. We start by formally defining the $n$-prism graph $\cP_n$.

\begin{definition}
    Let $n$ be a positive integer. The  $n$-prism graph $\cP_n$ is the graph on $2n$ vertices such that $V(\cP_n)=\{ u_i, v_i \ |\ i=1,\dots, n \}$, where
\begin{enumerate}
	\item[$(1)$] $(u_1 u_2 \dots u_{n} )$ and $(v_1 v_2 \dots v_{n} )$ are two disjoint cycles of length $n$,
	\item[$(2)$] $u_i$ is adjacent to $v_i$, for each $i=1,\dots,n$.		
\end{enumerate}
\end{definition}

The $n$-cycle induced by the sets of vertices $\{u_{i} \ |\ i=1,\dots, n\}$ is referred to as the \emph{outer-cycle}, whilst the $n$-cycle induced by the vertices $\{v_{i}\ |\ i=1,\dots, n\}$ is referred to as the \emph{inner-cycle}. The edges on these two $n$-cycles are said to be the \emph{outer-edges} and \emph{inner-edges} accordingly, whilst the edges $u_{i}v_{i}$ are referred to as \emph{spokes}.

\begin{theorem}
	Let $n$ be a positive integer. Then, the only PMH $n$-prism graph $\cP_n$ is the Cube graph, i.e. $\cP_4$.
\end{theorem}

\begin{proof}
We start by noticing that when $n$ is odd, the inner and the outer cycle of the graph have odd length $n$, and so there is a 2-factor of $\cP_n$ that contains an odd cycle, hence for Corollary \ref{cor} it cannot be PMH.\\
If $n=4$, the Graph $\cP_4$ is the \textit{Cube Graph} $Q_3$, and it was shown in \cite{F07}  that it is \textit{PMH}.\\
Thus, we may suppose $n$ even, $n \geq 6$ and consider the following perfect matching, say $M$, of $\cP_n$ built as follows: we take the edges $u_1v_1,u_2v_2$ and then we alternate, as far as possible, a pair of parallel edges and a spoke. We have the following 3 different possibilities, depending on the parity of $n$ modulo 6:
\begin{itemize}
	\item[1)] if $n \equiv 0 \ (\mod \ 6)$: there will be 4 consecutive spokes, that are $u_1v_1, u_2v_2,$ $ u_nv_n$ and $u_{n-1}v_{n-1}$,
	\item[2)] if $n \equiv 2 \ (\mod \ 6)$: there will be 3 consecutive spokes, that are $u_1v_1, u_2v_2$ and $u_nv_n$,
	\item[3)] if $n \equiv 4 \ (\mod \ 6)$: there will be 2 consecutive spokes, that are $u_1v_1$ and $u_2v_2$.
\end{itemize}

We will show below that the perfect matching $M$ of $\cP_n$, when $n$ even and $n \geq 6$, defined as above cannot be extended to a Hamiltonian cycle.\\
In fact, suppose on the contrary that $M$ can be extended to a Hamiltonian Cycle $H$. The edges $u_1v_1, u_2v_2 \in M$, so they also belong to $H$. 
Let $u_1u_2 \in E(H)$, then $u_3v_3 \in E(H)$ and also $v_1v_2 \in E(H)$, but $(u_1v_1v_2u_2u_1)$ is a $4$-cycle in $H$.\\
Now suppose $u_1u_n \in E(H)$. We have 3 possible cases:
\begin{enumerate}
	\item[(1)]  $n \equiv 0 \ (\mod \ 6)$: then $u_1v_1, u_2v_2, u_nv_n,u_{n-1}v_{n-1} \in M$.
	If $v_1v_2 \in E(H)$ then $v_3u_3 \in E(H)$. But this forces $u_1u_2 \in E(H)$, a contradiction. So $v_1v_n \in E(H)$, but this forces the $4$-cycle $(v_nv_1u_1u_nv_n)$ to be in $H$, a contradiction.
	\item[(2)] $n \equiv 2 \ (\mod \ 6)$: then  $u_1v_1, u_2v_2,u_nv_n \in M$. Since $u_1u_n \in E(H)$, then $u_{n-1}v_{n-1} \in E(H)$. But this forces the $4$-cycle $(u_1v_1v_nu_nu_1)$ to be in $H$, a contradiction.
	\item[(3)] $n \equiv 4 \ (\mod \ 6)$: then $u_1v_1, u_2v_2 \in M$. If $v_1v_2 \in E(H)$ then $v_3u_3 \in E(H)$. But this forces $u_1u_2 \in E(H)$, a contradiction. So $v_1v_n \in E(H)$. If $u_{n-1}v_{n-1} \in E(H)$, this forces the $6-$cycle $(u_1v_1v_nv_{n-1}u_{n-1}u_nu_1)$ to be in $H$, a contradiction. Thus $v_{n-1}v_{n-2}$ and $u_{n-1}u_{n-2} \in E(H)$, but this forces the $8$-cycle  $(u_1v_1v_nv_{n-1}v_{n-2}u_{n-2}u_{n-1}u_nu_1)$ to be in $H$, a contradiction.

\end{enumerate}

Therefore, $M$ belongs to no Hamiltonian Cycles in $\cP_n$ when $n\geq 6$. 
\end{proof}

\begin{figure}[h!]
\centering
\includegraphics[width=1\textwidth]{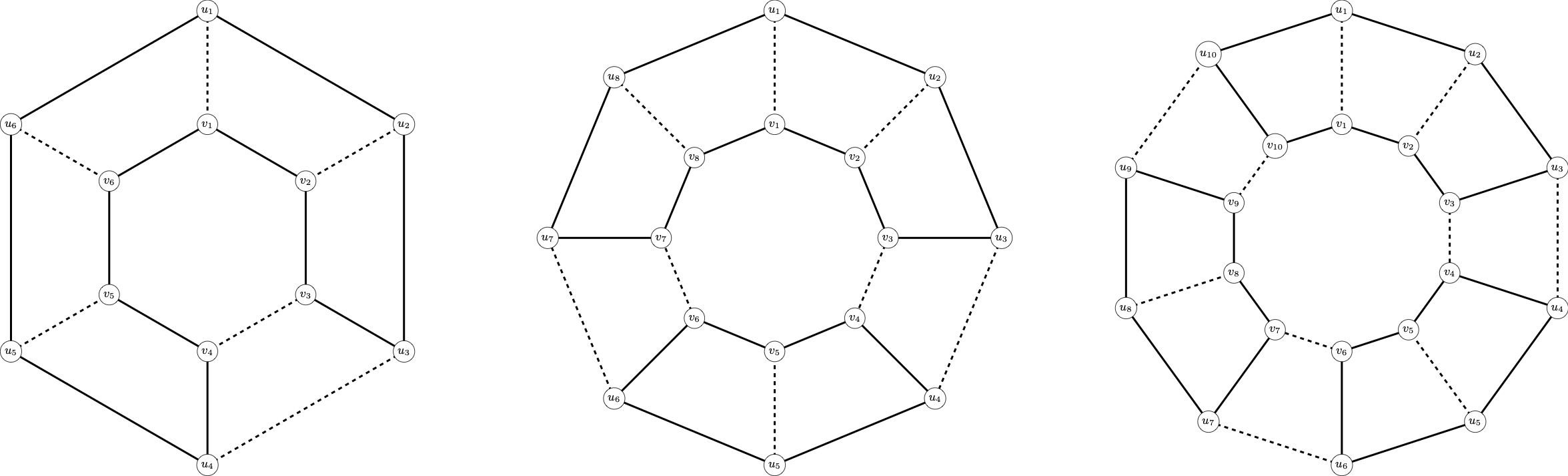}
\caption{The perfect matching $M$, bold dashed edges, in $\cP_6$ (left), in $\cP_8$ (center), and in $\cP_{10}$ (right).}
\label{figure 2}
\end{figure}

\newpage

\section{The PMH-property in Crossed prism graphs}\label{CPn}
As mentioned in the introduction, this section will focus on studying the $PMH$-property in Crossed Prism graphs. As above, we first formally define the $2n$-crossed prism graph $\cCP_n$.

\begin{definition}
    Let $n$ be a positive integer. The  $2n$-crossed prism graph $\cCP_n$ is the graph on $8n$ vertices such that $V(\cCP_n)=\{ u_i, v_i \ |\ i=1,\dots, 4n \}$, where
\begin{enumerate}
	\item[$(1)$] $(u_1 u_2 \dots u_{4n} )$  is a cycle of length $4n$,
	\item[$(2)$] $u_i$ is adjacent to $v_i$, for each $i=1,\dots,4n$,
	\item[$(3)$] the adjacencies between the vertices $v_i$, for $i=1,\dots,4n$, form a cycle of length $4n$ given by the edge set 
	\[
	\{v_{2i-1}v_{2i} \ | \ i=1,\dots,2n \}\cup \{v_{2i-1}v_{2i+2} \ | \ i=1,\dots,2n-1 \} \cup \{v_{4n-1}v_2\}.
	\]
\end{enumerate}
\end{definition}

The $4n$-cycle induced by the sets of vertices $\{u_{i} \ |\ i=1,\dots, 4n\}$ is referred to as the \emph{outer-cycle}, whilst the $4n$-cycle induced by the vertices $\{v_{i}\ |\ i=1,\dots, 4n\}$ is referred to as the \emph{inner-cycle}. The edges on these two $4n$-cycles are said to be the \emph{outer-edges} and \emph{inner-edges} accordingly, whilst the edges $u_{i}v_{i}$ are referred to as \emph{spokes}. The edges $u_{1}u_{4n}$, $v_{2}v_{4n-1}$, $v_{2n-1}v_{2n+2}$, $u_{2n}u_{2n+1}$ are denoted by $a,b,c,d$, respectively, and we shall also denote the set $\{a,b,c,d\}$ by $\mathcal{X}$. The set $\mathcal{X}$ is referred to as the \emph{principal $4$-edge-cut} of $\cCP_n$ (see Figure \ref{figure 1}).

\begin{figure}[h!]
\centering
\includegraphics[width=1\textwidth]{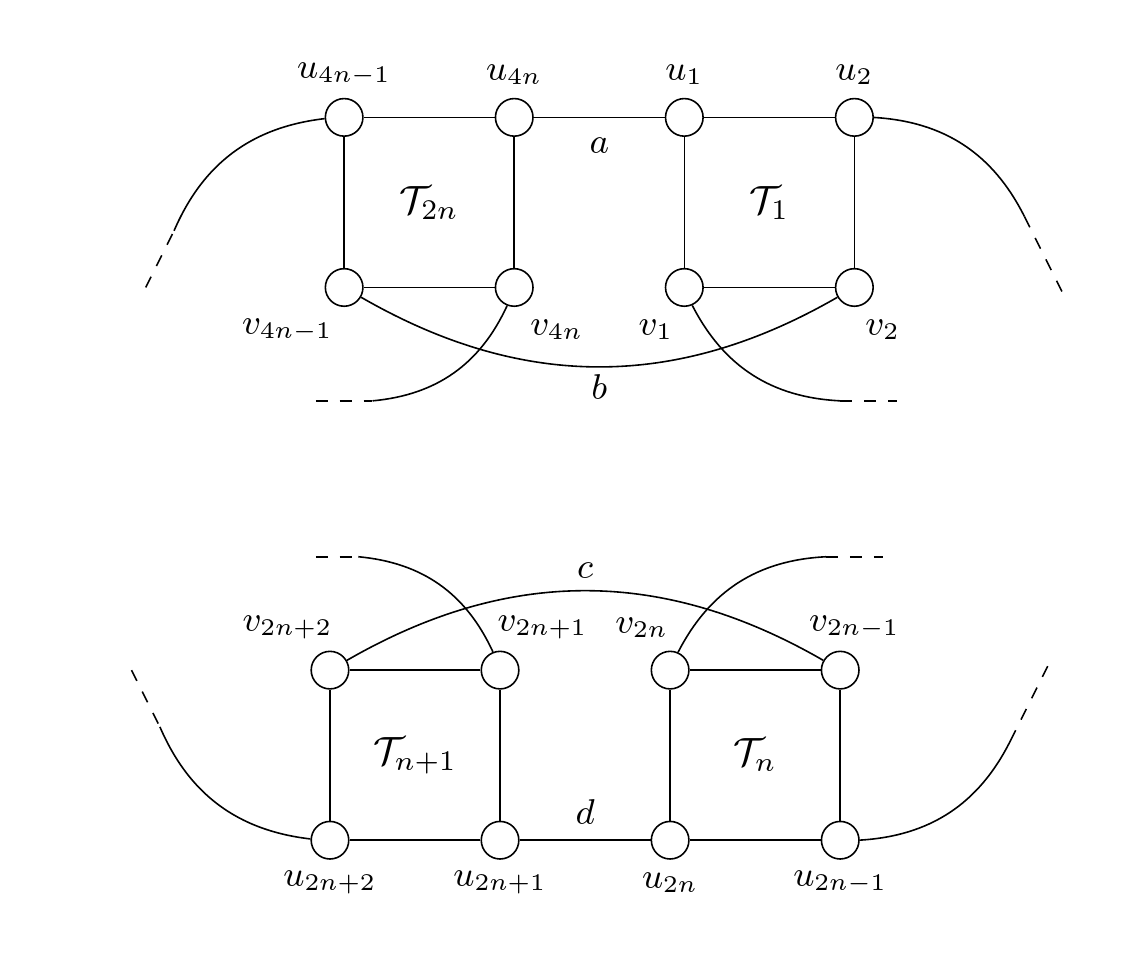}
\caption{The $2n$-crossed prism graph $\cCP_n$ and the $4-$pole $\cT_j$, for $j=1,\dots,2n$.}
\label{figure 1}
\end{figure}

In what follows we will use the same notation used in \cite[Section 2]{papillon}, as it is easy to see that Crossed Prism graphs almost have the same structure of Balanced Papillon graphs, except for the principal $4$-edge-cut.\\


A \emph{multipole} $\mathcal{Z}$ is defined by a set of vertices $V(\mathcal{Z})$ and a collection of generalised edges, where each generalised edge is either a standard edge (with two endvertices) or a semiedge. A \emph{semiedge} is a generalised edge with exactly one endvertex. The set of semiedges of $\mathcal{Z}$ is denoted by $\delta\mathcal{Z}$, while the set of standard edges (with two endvertices) is represented by $E(\mathcal{Z})$. Two semiedges are said to be \emph{joined} if they are both removed and their respective endvertices are made adjacent. A \emph{$k$-pole} refers to a multipole with $k$ semiedges. A perfect matching $M$ of a $k$-pole $\mathcal{Z}$ is a subset of the generalised edges of $\mathcal{Z}$ such that every vertex of $\mathcal{Z}$ is incident with exactly one generalised edge in $M$.

In what follows, we shall construct Crossed Prism graphs by joining the semiedges of various multipoles. Specifically, given a perfect matching $M$ of a graph $\cCP_n$ and a multipole $\mathcal{Z}$ used as a building block for constructing $\cCP_n$, we say that $M$ contains a semiedge $e$ of the multipole $\mathcal{Z}$ if $M$ includes the edge in $\cCP_n$ that results from joining $e$ to another semiedge during the construction of $\cCP_n$.

The $4$-pole $\mathcal{Z}$ with vertex set $\{z_{1}, z_{2},z_{3},z_{4}\}$,  such that $E(\mathcal{Z})$ induces the $4$-cycle $(z_{1}, z_{2},z_{3},z_{4})$ and with exactly one semiedge incident to each of its vertices is referred to as a \emph{$C_{4}$-pole} (see Figure \ref{figure 3}). For each $i=1,\dots,4$, let the semiedge incident to $z_{i}$ be denoted by $f_{i}$. The semiedges $f_{1}$ and $f_{2}$ are referred to as the \emph{upper left semiedge} and the \emph{upper right semiedge} of $\mathcal{Z}$, respectively. On the other hand, the semiedges $f_{3}$ and $f_{4}$ are referred to as the \emph{lower left semiedge} and the \emph{lower right semiedge} of $\mathcal{Z}$, respectively (see Figure \ref{figure 3}).

\begin{figure}[h!]
\centering
\includegraphics[width=0.9\textwidth]{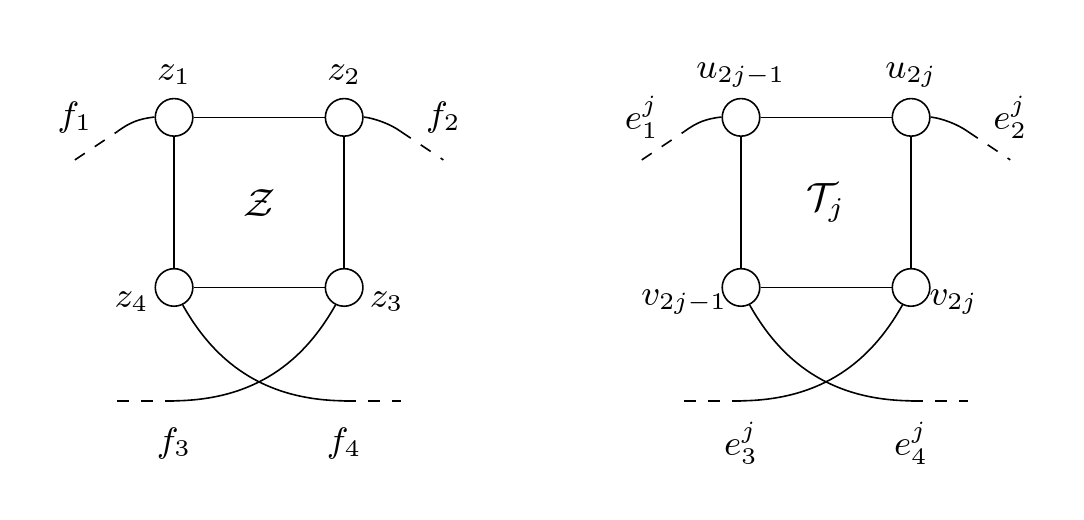}
\caption{A $C_4-$pole $\mathcal{Z}$ and the $\cT_j-$pole in $\cCP_n.$}
\label{figure 3}
\end{figure}

For some integer $m\geq 1$, let $\mathcal{Z}_{1},\ldots, \mathcal{Z}_{m}$ be $m$ copies of the above $C_4$-pole $\mathcal{Z}$. For each $j=1,\dots,m$, let $V(\mathcal{Z}_{j})=\{z_{1}^{j}, z_{2}^{j},z_{3}^{j},z_{4}^{j}\}$, and let $f_1^j,f_2^j,f_3^j,f_4^j$ be the semiedges of $\mathcal{Z}_{j}$ respectively incident to $z_{1}^{j}, z_{2}^{j},z_{3}^{j},z_{4}^{j}$ such that $f_{1}^{j}$ and $f_{2}^{j}$ are the upper left and upper right semiedges of $\mathcal{Z}_{j}$, whilst $f_{3}^{j}$ and $f_{4}^{j}$ are the lower left and lower right semiedges of $\mathcal{Z}_{j}$. A \emph{chain of $C_{4}$-poles} of length $m\geq 2$, is the $4$-pole obtained by respectively joining $f_{2}^{j}$ and $f_{4}^{j}$ (upper and lower right semiedges of $\mathcal{Z}_{j}$) to $f_{1}^{j+1}$ and $f_{3}^{j+1}$ (upper and lower left semiedges of $\mathcal{Z}_{j+1}$), for every $j=1,\dots,m-1$. When $m=1$, a chain of $C_{4}$-poles of length $1$ is just a $C_{4}$-pole. For simplicity, we shall refer to a chain of $C_{4}$-poles of length $m$, as a \emph{$m$-chain of $C_{4}$-poles}, or simply a \emph{$m$-chain}. The semiedges $f_{1}^{1}$ and $f_{3}^{1}$ (similarly, $f_{2}^{m}$ and $f_{4}^{m}$) are referred to as the upper left and lower left (respectively, upper right and lower right) semiedges of the $m$-chain.
A chain of $C_{4}$-poles of any length has exactly four semiedges. For simplicity, when we say that $e_{1},e_{2},e_{3},e_{4}$ are the four semiedges of a chain $\mathcal{Z}'$ of $C_{4}$-poles (possibly of length $1$), we mean that $e_{1}$ and $e_{2}$ are respectively the upper left and upper right semiedges of $\mathcal{Z}'$, whilst $e_{3}$ and $e_{4}$ are respectively the lower left and lower right semiedges of the same chain $\mathcal{Z}'$ (see Figure \ref{figure aa}). The semiedges $e_{1}$ and $e_{2}$ (similarly, $e_{3}$ and $e_{4}$) are referred to collectively as the \emph{upper semiedges} (respectively, \emph{lower semiedges}) of $\mathcal{Z}'$. In a similar way, the semiedges $e_{1}$ and $e_{3}$ (similarly, $e_{2}$ and $e_{4}$) are referred to collectively as the \emph{left semiedges} (respectively, \emph{right semiedges}) of $\mathcal{Z}'$.

\begin{figure}[h!]
\centering
\includegraphics[width=0.9\textwidth]{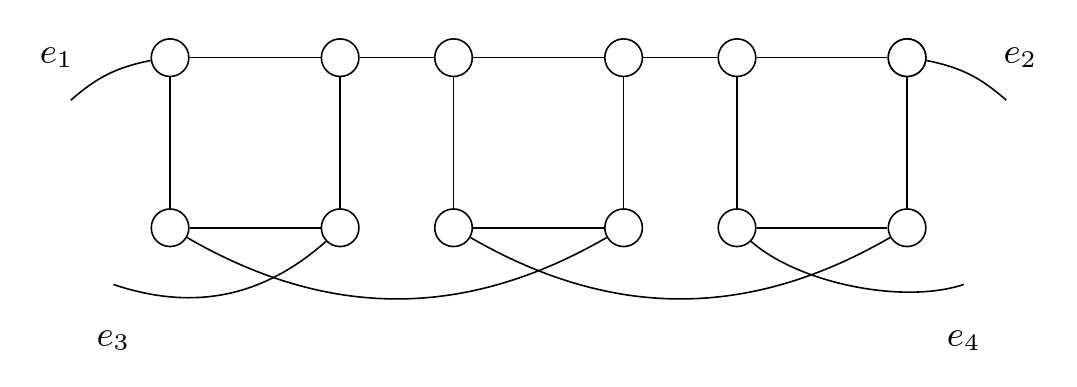}
\caption{A chain of $C_4-$poles of length 3 having semiedges $e_1, e_2,e_3,e_4.$}
\label{figure aa}
\end{figure}

Let $M$ be a perfect matching of the $2n$-crossed prism graph $\cCP_n$. Since $\cX=\{a,b,c,d\}$ is a $4$-edge-cut of $\cCP_n$, $|M\cap \cX|\equiv 0\pmod{2}$, that is, $|M\cap \cX|$ is $0, 2$ or $4$. The following is a useful lemma which shall be used frequently in the results that follow.

\begin{lemma}\label{lem1}
	Let M be a perfect matching of the $2n$-crossed prism graph $\cCP_n$ and let $\cX$ be its principal $4$-edge-cut. If $|M \cap \cX|=k$, then $|M \cap  \delta\cT_j|=k$ for each $j=1,\dots,2n$.
\end{lemma}

\begin{proof}
Let $M$ be a perfect matching of the $2n$-crossed prism graph $\cCP_n$. We first note that the left semiedges of a $C_4-$pole are contained in the same perfect matching. The lemma is proved by considering three cases depending on the possible values of $k$, that is 0,2 or 4.

\noindent\textbf{Case 1:} $k=0$.

Since $a$ and $b$ do not belong to $M$, the left semiedges of $\cT_1$ aren't contained in $M$, and so $M$ cannot contain its right semiedges. Therefore $|M \cap \delta \cT_1|=0$. Consequently, the left semiedges of $\cT_2$ aren't contained in M implying again that $|M \cap \delta\cT_2|=0$. By repeating the same argument up till the $n^{th}$ $C_4-$pole, we have $|M \cap \delta \cT_j|=k$, for every $j=1,\dots,n$. By noting that $c$ and $d$ don't belong to $M$ and repeating a similar argument to the 4$-$poles in the left $n-$chain, we can deduce that $|M \cap \delta \cT_j|=k$, for every $j=1,\dots,2n$.
		
\noindent\textbf{Case 2: } $k=4$.
  
Since $a$ and $b$ belong to $M$, the left semiedges of $\cT_1$ are contained in $M$, and so $M$  contains its right semiedges as well. Therefore,  $|M \cap \delta \cT_1|=4$. Consequently, the left semiedges of $\cT_2$ are contained in $M$ implying again that $|M \cap \delta \cT_2|=4$. As in Case 1, by noting that $c$ and $d$ belong to $M$ and repeating a similar argument to the 4$-$poles in the left $n-$chain, we can deduce that $|M \cap \delta \cT_j|=4$, for every $j=1,\dots,2n$.

\noindent\textbf{Case 3:} $k=2$.	 
  
In this case, we first claim that $M \cap \cX$ must be equal to $\{a,d\}$ or $\{a,c\}$ or $\{b,c\}$ or $\{b,d\}$. In fact, suppose that $M\cap \mathcal{X}=\{a,b\}$, without loss of generality. This means that the right semiedges of $\cT_1$ are also contained in $M$, implying that $|M\cap \cT_{1}|=4$. This implies that the left semiedges of $\cT_2$ are contained in $M$, which forces $|M\cap \cT_{j}|$ to be equal to 4, for every $j=1,\dots,2n$. In particular, $|M\cap \delta \cT_{n}|=4$, implying that the edges $c$ and $d$ belong to $M$, a contradiction since $M\cap \cX=\{a,b\}$. This proves our claim. 
The natural automorphism $\psi$ of $\cCP_n$, which exchanges the outer- and inner-cycles, exchanges also $\{a,d\}$ with $\{b,c\}$, combined with the mirror automorphism through the horizontal axis allow us to assume, without loss of generality, that $M \cap \mathcal{X}=\{a,d\}$. Since $b \notin M$, $1 \leq |M \cap \delta \cT_1|\leq 3$.  But $\delta \cT_1$ corresponds to a 4-edge cut in $\cCP_n$, and so, by using a parity argument, $ |M \cap \delta \cT_1|$ must be equal to 2, implying that exactly one of the right semiedges of $\cT_1$ is contained in $M$.  This means that exactly one left semiedge of $\cT_{2}$ is contained in $M$, and consequently, by a similar argument now applied to $\cT_{2}$, we obtain $|M\cap \cT_{2}|=2$. By repeating the same argument and noting that $\cT_{n+1}$ has exactly one left semiedge (corresponding to the edge $d$) contained in $M$, one can deduce that $|M\cap \cT_{j}|=2$ for every $j=1,\dots,2n$.

\end{proof}

The following two results are two consequences of the above lemma and they both follow directly from its proof.

\begin{corollary}\label{cor1}
	Let $M$ be a perfect matching of $\cCP_n$ and let $\mathcal{X}$ be its principal $4$-edge-cut. If $|M\cap \mathcal{X}|=2$, then $M\cap \cX$ is equal to $\{a,d\}$ or $\{a,c\}$ or $\{b,c\}$ or $\{b,d\}$.
\end{corollary}

\begin{corollary}\label{cor2}
	Let $M$ be a perfect matching of $\cCP_n$ and let $\mathcal{X}$ be its principal $4$-edge-cut such that $|M\cap \mathcal{X}|=2$. For each $j=1,\dots,2n$, $M$ contains exactly one of the following sets of semiedges: $\{e_1^j,e_2^j\}, \{e_3^j,e_4^j\}, \{e_1^j,e_4^j\}, \{e_2^j,e_3^j\}$, that is, of all possible pairs of semiedges of $\cT_{j}$, $\{e_1^j,e_3^j\}$ and $\{e_2^j,e_4^j\}$ cannot be contained in $M$.
\end{corollary}


In the sequel, the process of traversing one path after another shall be referred to as the \emph{concatenation of paths}. When two paths $P$ and $Q$ have endvertices $x, y$ and $y, z$, respectively, we denote the concatenated path as $PQ$, representing the path that starts at $x$ and ends at $z$ by traversing $P$ followed by $Q$. Without loss of generality, if $x$ is adjacent to $y$, meaning $P$ is a path on two vertices, we may denote this as $xyQ$ instead of $PQ$.

\begin{lemma}\label{lemma complementary}
	Let $M_{1}$ be a perfect matching of $\cCP_n$ such that $|M_{1}\cap \cX|=2$.
	\begin{enumerate}
		\item[$(1)$] There exists a perfect matching $M_2$ of $\cCP_n$ such that 
		$|M_{2}\cap \mathcal{X}|=2$ and \\ $M_1\cap M_2=\emptyset$.
		\item[$(2)$] The complementary $2$-factors of $M_1$ and $M_2$ are both Hamiltonian cycles.
	\end{enumerate}
\end{lemma}
\begin{proof}
	(1)\,\, Since $|M_{1}\cap \mathcal{X}|=2$, by Lemma \ref{lem1} we get that $|M_{1}\cap \delta \cT_{j}|=2$ for every $j=1,\dots,2n$. For each $j$, let $P^{(j)}$ be the subgraph of $\cCP_n$ which is induced by $E(\cT_{j})\setminus M_{1}$. Note that $\bigcup\limits_{j=1}^{2n} V(P^{(j)})=V(G)$. By Corollary \ref{cor2}, each $P^{(j)}$ is a path of length 3. Letting $N$ be the unique perfect matching of $\cCP_n$ which intersects each $E(P^{(j)})$ in exactly two edges, we note that $M_{1} \cap N=\emptyset$. Let $M_{2}=E(G) \setminus (M_{1}\cup N)$. Since $M_{1}$ and $N$ are two disjoint perfect matchings, $M_{2}$ is also a perfect matching of $\cCP_n$ and, in particular, $M_{2}$ contains $\mathcal{X}\setminus (M_{1}\cap \mathcal{X})$. Thus, $|M_{2}\cap \mathcal{X}|=2$ and $M_{1}\cap M_{2}=\emptyset$, proving part $(1)$.

	(2)\,\, Let $M_{2}$ be as in part $(1)$, that is, $|M_{2}\cap \mathcal{X}|=2$ and $M_{1}\cap M_{2}=\emptyset$. For $i,j=1,\dots,2n$ with $i \not=j$, let $Q^{(i,j)}$ be the subgraph of $\cCP_n$ which is induced by 
	\[
	M_{2}\cap \{xy\in E(G): x\in V(\cT_{i}),\ y\in V(\cT_{j})\},
	\]
	that is, $E(Q^{(i,j)})$ is either empty or consists of exactly one edge, that is, $Q^{(i,j)}$ is a path of length 1.	 
	 When $M_{1}\cap \mathcal{X}=\{a,d\}$, we can form a Hamiltonian cycle of $\cCP_n$ (not containing $M_{1}$) by considering the following concatenation of paths:
	\begin{linenomath}
		$$P^{(1)}Q^{(1,2)}\ldots Q^{(n-1,n)}P^{(n)}Q^{(n,n+1)}P^{(n+1)}Q^{(n+1,n+2)}\ldots P^{(2n)}Q^{(2n,1)},$$
	\end{linenomath}
	where $Q^{(1,2)}$ and $Q^{(n+1,n+2)}$ are respectively followed by $P^{(2)}$ and $P^{(n+2)}$, and, $Q^{(n,n+1)}$ and $Q^{(2n,1)}$ consist of the edges $c$ and $b$, respectively. \\
	When $M_{1}\cap \mathcal{X}=\{a,c\}$, we can form a Hamiltonian cycle of $\cCP_n$ (not containing $M_{1}$) by considering the following concatenation of paths:
	\begin{linenomath}
	$$P^{(1)}Q^{(1,2)}\ldots Q^{(n-1,n)}P^{(n)}Q^{(n,n+1)}P^{(n+1)}Q^{(n+1,n+2)}\ldots P^{(2n)}Q^{(2n,1)},$$
\end{linenomath}
where $Q^{(1,2)}$ and $Q^{(n+1,n+2)}$ are respectively followed by $P^{(2)}$ and $P^{(n+2)}$, and, $Q^{(n,n+1)}$ and $Q^{(2n,1)}$ consist of the edges $d$ and $b$, respectively. \\
When $M_{1}\cap \mathcal{X}=\{b,c\}$, we can form a Hamiltonian cycle of $\cCP_n$ (not containing $M_{1}$) by considering the following concatenation of paths:
	\begin{linenomath}
	$$P^{(1)}Q^{(1,2)}\ldots Q^{(n-1,n)}P^{(n)}Q^{(n,n+1)}P^{(n+1)}Q^{(n+1,n+2)}\ldots P^{(2n)}Q^{(2n,1)},$$
\end{linenomath}
where $Q^{(1,2)}$ and $Q^{(n+1,n+2)}$ are respectively followed by $P^{(2)}$ and $P^{(n+2)}$, and, $Q^{(n,n+1)}$ and $Q^{(2n,1)}$ consist of the edges $d$ and $a$, respectively.\\
Finally, when $M_{1}\cap \mathcal{X}=\{b,d\}$, we can form a Hamiltonian cycle of $\cCP_n$ (not containing $M_{1}$) by considering the following concatenation of paths:
	\begin{linenomath}
	$$P^{(1)}Q^{(1,2)}\ldots Q^{(n-1,n)}P^{(n)}Q^{(n,n+1)}P^{(n+1)}Q^{(n+1,n+2)}\ldots P^{(2n)}Q^{(2n,1)},$$
\end{linenomath}
where $Q^{(1,2)}$ and $Q^{(n+1,n+2)}$ are respectively followed by $P^{(2)}$ and $P^{(n+2)}$, and, $Q^{(n,n+1)}$ and $Q^{(2n,1)}$ consist of the edges $c$ and $a$, respectively.

Thus, the complementary 2-factor of $M_{1}$ is a Hamiltonian cycle. This is depicted in Figure \ref{figure 4}. The proof that the complementary 2-factor of $M_{2}$ is a Hamiltonian cycle follows analogously.
\end{proof}

\begin{figure}[h!]
\centering
\includegraphics[scale=0.3]{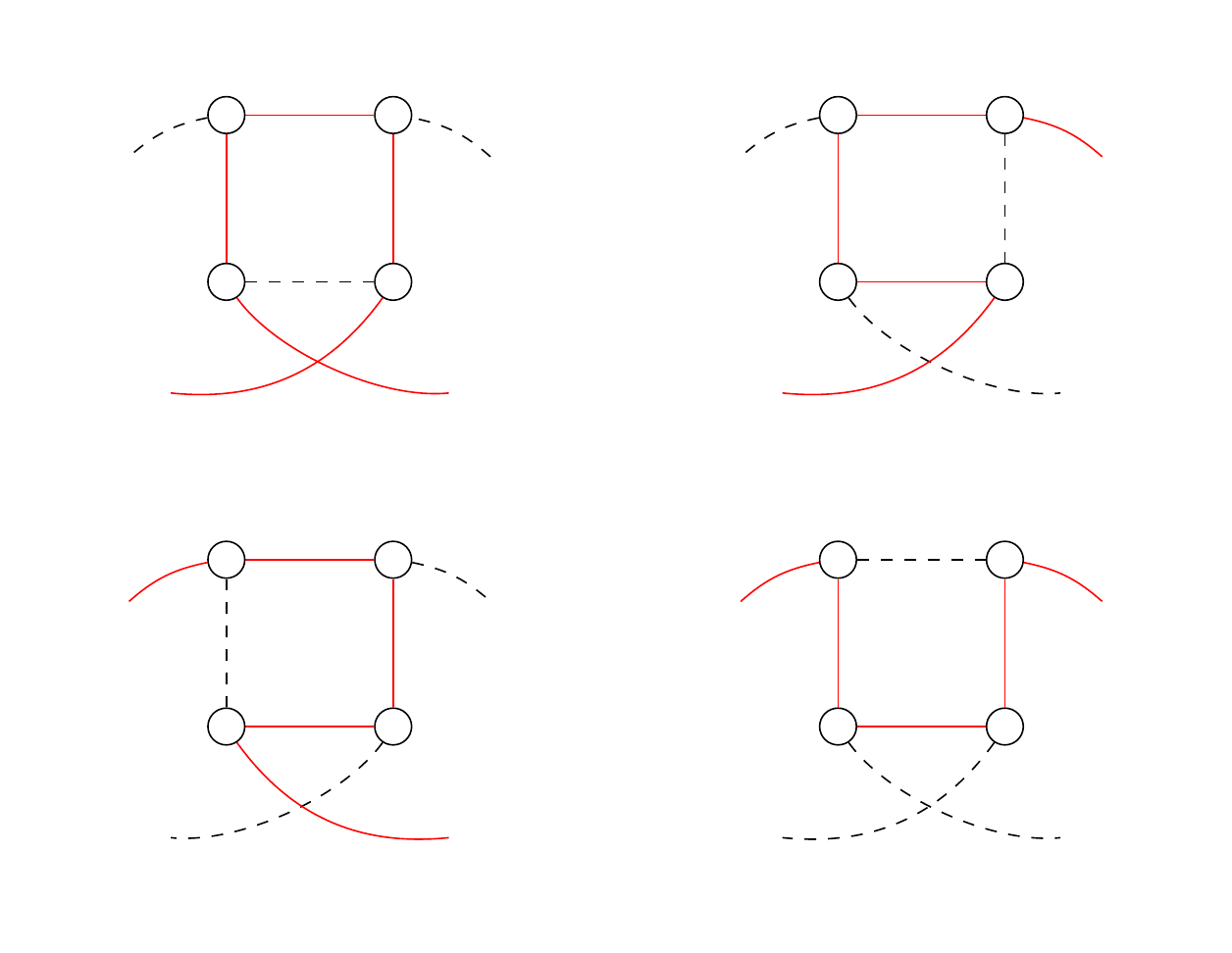}
\caption{Perfect Matching $M_1$ (bold dashed edges) with $|M_1 \cap \cX|=2$ and its complementary 2-factor (highlighted edges).}
\label{figure 4}
\end{figure}

\newpage

\begin{proposition}
	Let $n$ be a positive odd integer. Then, the $2n$-crossed prism graph $\cCP_n$ is not PMH.
\end{proposition}

\begin{proof}
	
	Consider the following perfect matching of the $2n$-crossed prism graph $\cCP_n$:
	\[
	M=\bigcup\limits_{i=1}^{2n}\{u_{2i-1}u_{2i}, v_{2i-1}v_{2i} \}.
	\]
	We claim that $M$ cannot be extended to a Hamiltonian cycle of $\cCP_n$. For, let $F$ be a 2-factor of $\cCP_n$ containing $M$. Since $u_1u_2 \in M$ and $\cCP_n$ is cubic, $F$ contains exactly one of the following two edges: $u_1u_{4n}$ or $u_1v_1$. In the former case, the choice of $u_1u_{4n}$ forbids all the spokes of $\cCP_n$ to belong to $F$; hence, all the edges of the outer- and inner-cycle will belong to $F$. So we have two disjoint cycles each of length $4n$. In the latter case, $F$ must also contain all spokes $u_i v_i$, for $i=2,\dots,4n$. In fact, the subgraph induced by the set of spokes is exactly the complement of the 2-factor obtained in the former case. So, $F$ will consist of $2n$  disjoint 4-cycles.
\end{proof}

Let $M$ be a perfect matching of $\cCP_n$ with $|M \cap \mathcal{X}|=0$, which by Lemma \ref{lem1} implies that $|M\cap \delta \cT_{j}|=0$ for all $j=1,\dots,2n$. Now consider $j=1,\dots, 2n$ with $j \neq n, 2n$ and let ${\cT}_{(j,j+1)}$ denote a $2$-chain composed of ${\cT}_{j}$ and ${\cT}_{j+1}$.
We say that $\cT_{(j,j+1)}$ is \emph{symmetric with respect to $M$} if exactly one of the following occurs:
\begin{enumerate}
	\item[(1)] $\{u_{2j-1}v_{2j-1},u_{2j}v_{2j},u_{2j+1}v_{2j+1},u_{2j+2}v_{2j+2}\}\subset M$; or
	\item[(2)] $\{u_{2j-1}u_{2j},v_{2j-1}v_{2j},u_{2j+1}u_{2j+2},v_{2j+1}v_{2j+2}\}\subset M$.
\end{enumerate}
If neither (1) nor (2) occur, ${\cT}_{(j,j+1)}$ is said to be \emph{asymmetric with respect to $M$}. This is shown in Figure \ref{figure 5}.

\begin{figure}[h!]
\centering
\includegraphics[scale=0.25]{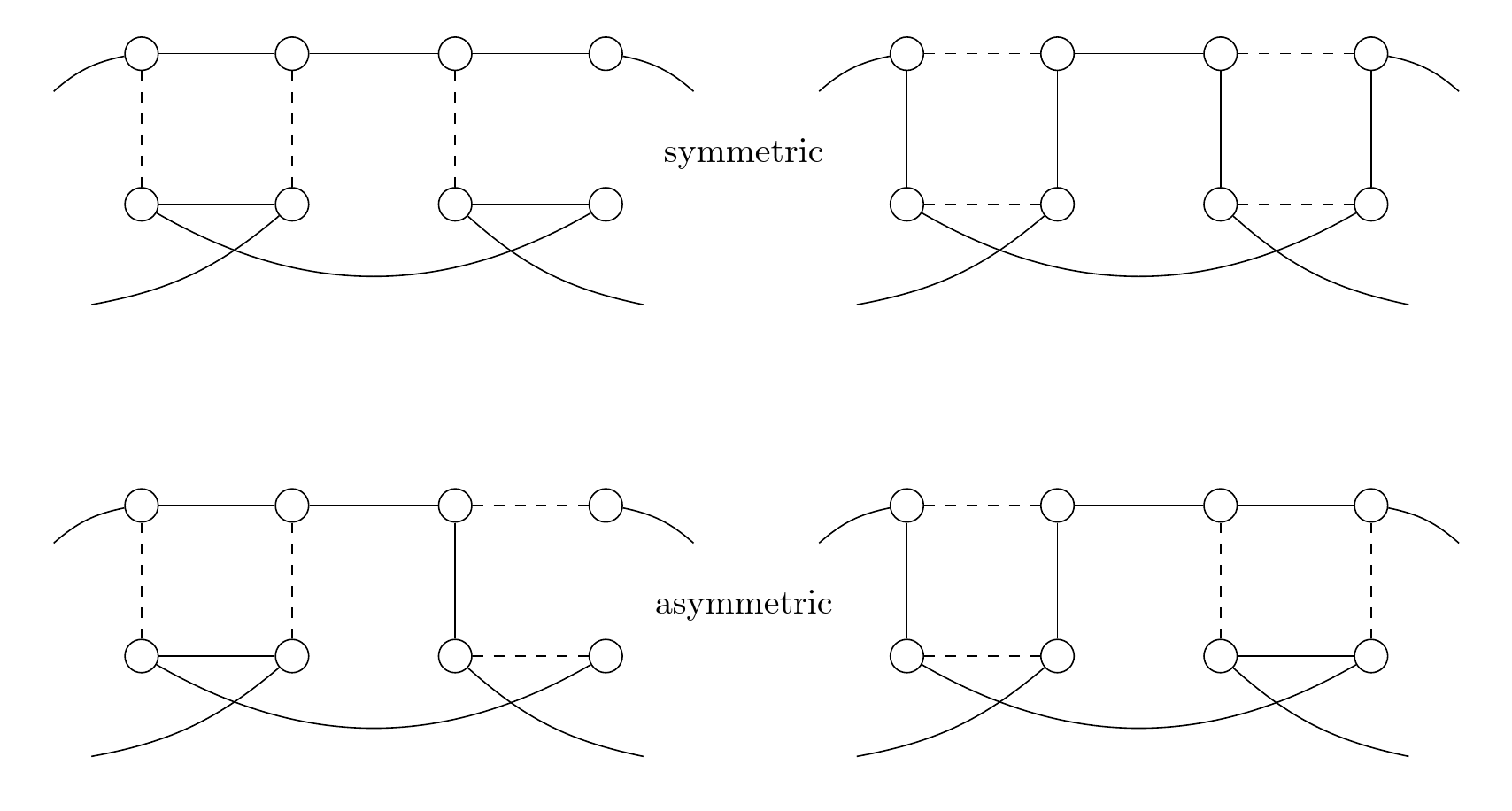}
\caption{Symmetric and asymmetric 2-chains with the dashed edges belonging to $M$.}
\label{figure 5}
\end{figure}

\begin{remark}\label{rem1}
{\em Consider a perfect matching $M_{1}$ of $\cCP_n$ such that $M_{1}$ does not intersect the principal $4$-edge-cut $\mathcal{X}$ of $\cCP_n$, that is, $M_{1}\cap \mathcal{X}=\emptyset$, and consider a $2$-chain of $\cCP_n$, say ${\cT}_{(j,j+1)}$ with $j=1,\dots, 2n$ with $j \neq n,2n$, having semiedges $e_1,e_2,e_3,e_4$, where $e_1=e_{1}^{j},e_2=e_{2}^{j+1},e_3=e_{3}^{j}$ and $e_4=e_{4}^{j+1}$.
Assume there exists a perfect matching  $M_{2}$ of $\cCP_n$ such that $|M_{2}\cap \mathcal{X}|=2$ and $M_{1}\cap M_{2}=\emptyset$ (see Figure \ref{figure remark1}). If $\mathcal{T}_{(j,j+1)}$ is symmetric with respect to $M_{1}$, then we have exactly one of the following instances:
\[
	M_{2} \cap \delta\cT_{(j,j+1)}=\{e_1,e_2\} \mbox{\ $($upper$)$; \quad or \quad} M_{2} \cap \delta{\cT}_{(j,j+1)}=\{e_3,e_4\} \mbox{\ $($lower$)$}.
\]
Otherwise, if ${\cT}_{(j,j+1)}$ is asymmetric with respect to $M_{1}$, then exactly one of the following must occur:
\[
	M_{2} \cap \delta{\cT}_{(j,j+1)}=\{e_1,e_4\} \mbox{ $(\text{upper left, lower right); or}$ }
\]
\[
	M_{2} \cap \delta{\cT}_{(j,j+1)}=\{e_2,e_3\} \mbox{ $(\text{upper right, lower left)}$}.
\]
Notwithstanding whether $\mathcal{T}_{(j,j+1)}$ is symmetric or asymmetric with respect to $M_{1}$, $(M_{1}\cup M_{2})\cap E(\mathcal{T}_{(j,j+1)})$ induces a path $($see Figure $\ref{figure remark1})$ which contains all the vertices of $V(\mathcal{T}_{(j,j+1)})$, and whose endvertices are the endvertices of the semiedges in $M_{2}\cap\partial\mathcal{T}_{(j,j+1)}$. 
}
\end{remark}

\begin{figure}[h!]
\centering
\includegraphics[scale=0.25]{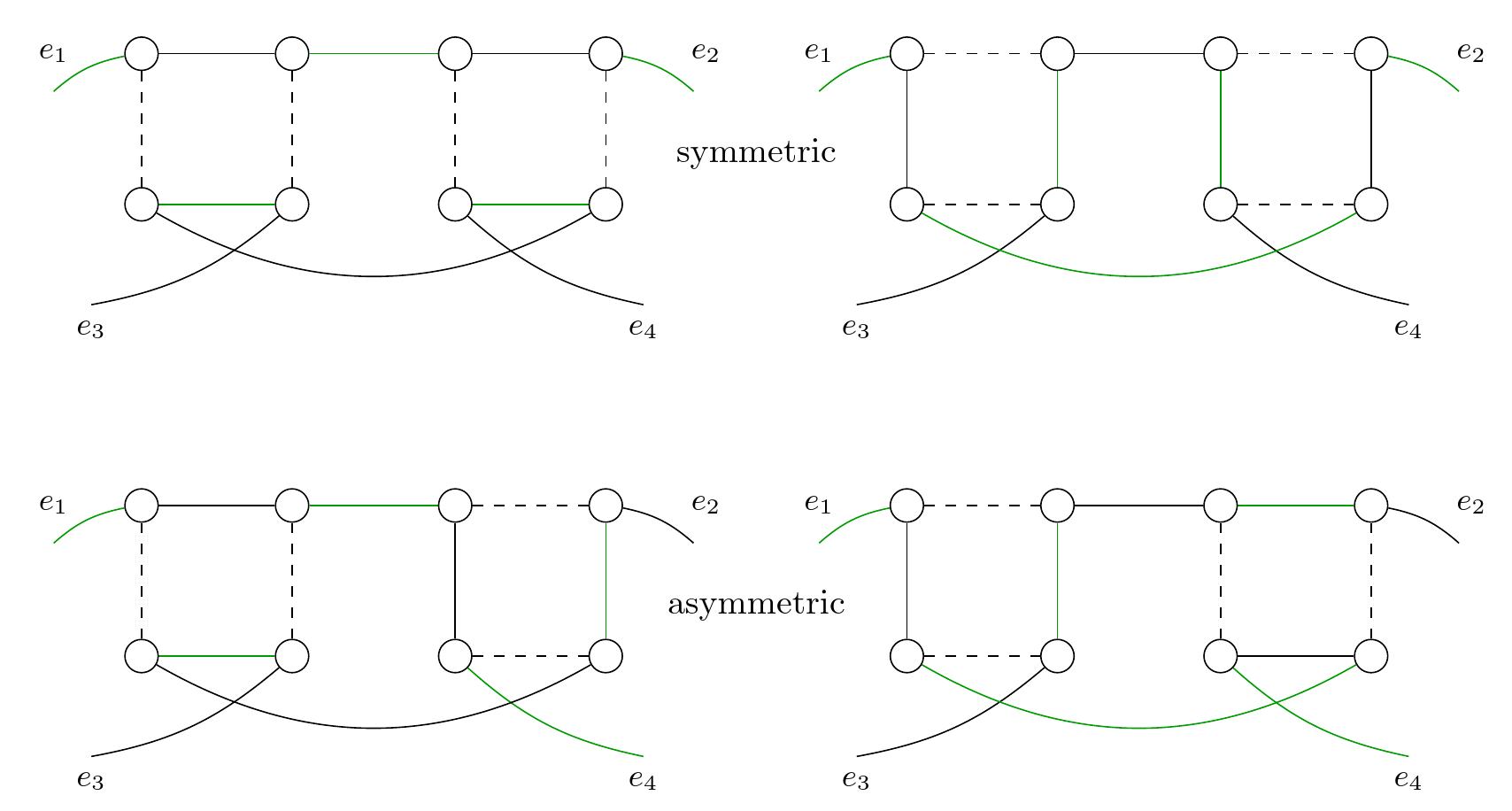}
\caption{2-chains when $M_1 \cap \cX = \emptyset$ and $|M_2 \cap \cX| =2$ (bold dashed edges belong to $M_1$ and 
highlighted to $M_2$).}
\label{figure remark1}
\end{figure}

\begin{remark}\label{rem2}
{\em Let $n\geq 2$. Consider a perfect matching $M_{1}$ of $\cCP_n$ such that $M_{1}$ does not intersect the principal $4$-edge-cut $\mathcal{X}$ of $\cCP_n$, that is, $M_{1}\cap \mathcal{X}=\emptyset$, and consider a $2$-chain of $\cCP_n$, say ${\cT}_{(j,j+1)}$ with $j=1,\dots,2n$ and $j \neq n,2n$.
Let $M_{2}$ be the perfect matching of $\cCP_n$ such that $|M_{2}\cap \mathcal{X}|=4$. Clearly $M_{1}\cap M_{2}=\emptyset$.\\ Notwithstanding whether $\cT_{(j,j+1)}$ is symmetric or asymmetric with respect to $M_{1}$, we have that $(M_{1}\cup M_{2})\cap E({T}_{(j,j+1)})$ induces two disjoint paths of equal length $($see Figure $\ref{figure remark2})$ whose union contains all the vertices of $\cT_{j}$ and $\cT_{j+1}$. Let $Q$ be one of these paths. We first note that $Q$ contains exactly one vertex from $\{u_{j},v_{j+1}\}$ and exactly one vertex from $\{u_{j+3},v_{j+2}\}$. If $\cT_{(j,j+1)}$ is symmetric with respect to $M_{1}$, then $Q$ contains $u_{j}$ if and only if $Q$ contains $u_{j+3}$. Otherwise, if $\cT_{(j,j+1)}$ is asymmetric with respect to $M_{1}$, then $Q$ contains $u_{j}$ if and only if $Q$ contains $v_{j+2}$.	
}
\end{remark}

\begin{figure}[h!]
\centering
\includegraphics[scale=0.25]{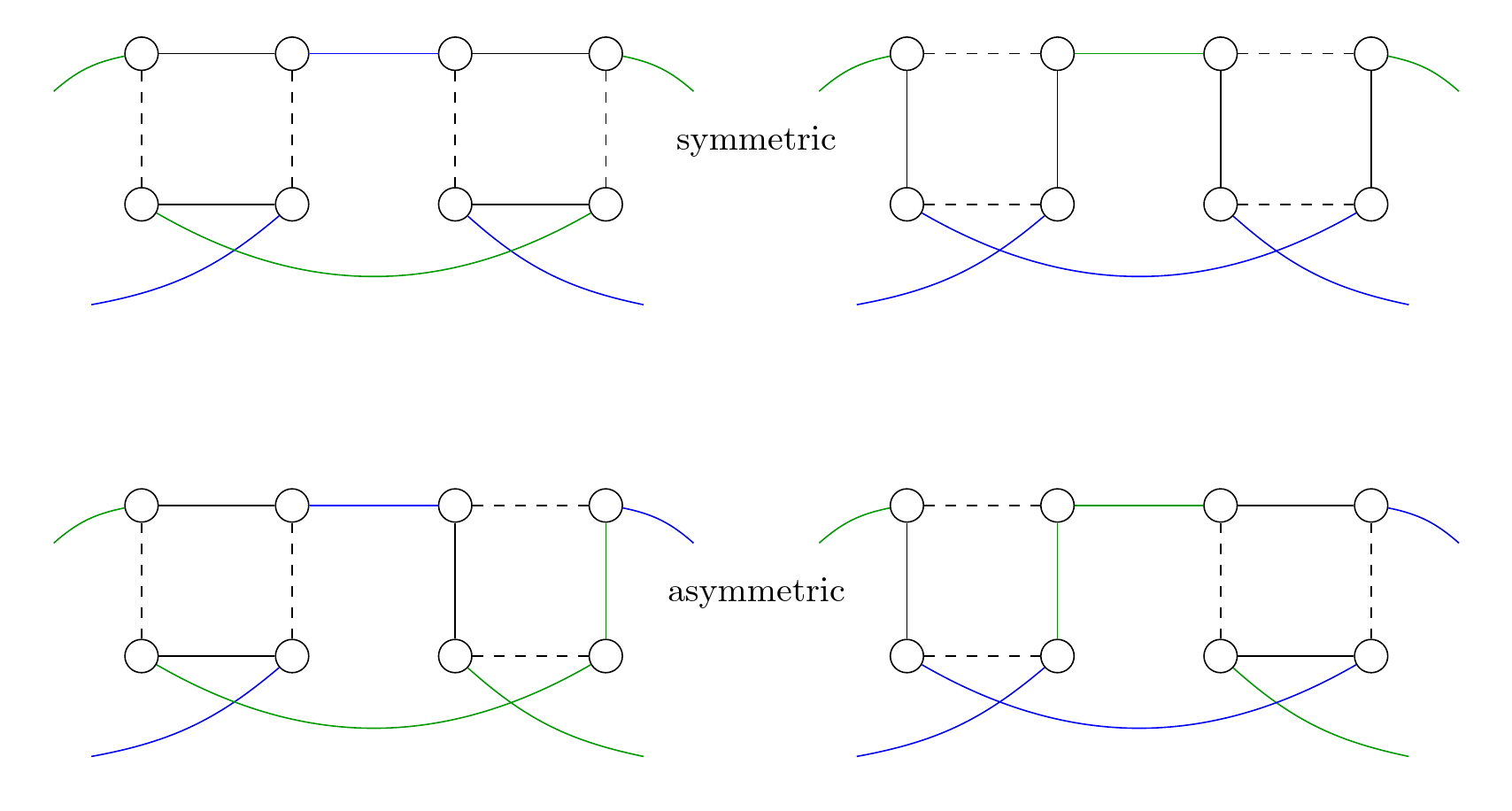}
\caption{2-chains when $M_1 \cap \cX = \emptyset$ and $|M_2 \cap \cX| =4$ (bold dashed edges belong to $M_1$ and highlighted to $M_2$).}
\label{figure remark2}
\end{figure}

\begin{theorem}\label{theorem pmh crossed prisma}
	Let $n$ be a positive even integer. Then, the $2n$-crossed prism graph $\cCP_n$ is PMH.
\end{theorem}

\begin{proof}
	Let $M_{1}$ be a perfect matching of $\cCP_n$. We need to show that there exists a perfect matching $M_{2}$ of $\cCP_n$ such that $M_{1}\cup M_{2}$ induces a Hamiltonian cycle of $\cCP_n$. Three cases, depending on the intersection of $M_{1}$ with the principal 4-edge-cut $\mathcal{X}$ of $\cCP_n$, are considered.

\noindent\textbf{Case 1:} $|M_{1}\cap \mathcal{X}|=2$.

   By Lemma \ref{lemma complementary}, there exists a perfect matching $N$ of $\cCP_n$ such that $|N\cap \mathcal{X}|=2$ and $M_{1}\cap N=\emptyset$. Moreover, the complementary 2-factor of $N$ is a Hamiltonian cycle. Since $M_{1}$ is contained in the mentioned 2-factor, the result follows.\\
   
\noindent\textbf{Case 2:} $|M_{1}\cap \mathcal{X}|=4$.
   
	In this case, we can define $M_{2}$ to be the following perfect matching:

	\[
		M_{2}= \{u_{1}v_{1},u_{2}v_{2}\} \cup \bigcup\limits_{j=2}^{2n} \{u_{2j-1}u_{2j},v_{2j-1}v_{2j}\} 
	\]
	
	In fact, $M_{1}\cup M_{2}$ induces the following Hamiltonian cycle:
	
	\[	 
	(u_{1},v_{1},v_{4},v_3, \ldots ,v_{2n}, v_{2n-1},v_{2n+2},\ldots,v_{4n-3},v_{4n},v_{4n-1},v_{2}, u_{2}, u_{3},u_{4},\ldots, u_{4n}),
	\]
	where $v_{3}$ and $v_{2n+2}$ are respectively followed by $v_6$ and $v_{2n+1}$.\\
	  
\noindent\textbf{Case 3:} $|M_{1}\cap \mathcal{X}|=0$.
    
	Clearly, $|M_{2}\cap \mathcal{X}|$ cannot be zero, because, if so, choosing $M_{2}$ to be disjoint from $M_{1}$, $M_{1}\cup M_{2}$ induces $2n$ disjoint $4$-cycles. Therefore, $|M_{2}\cap \mathcal{X}|$  must be equal to 2 or 4. 
   Let $\mathcal{R}=\{\cT_{(1,2)},\ldots, \cT_{(n-1,n)}\}$ and  $\mathcal{L}=\{\cT_{(n+1,n+2)},\ldots, \cT_{(2n-1,2n)}\}$ be the sets of 2-chains within the left and right $n$-chains of $\cCP_n$ namely the right and left $n$-chains each split into $\frac{n}{2}$ $2$-chains. We consider two cases depending on the parity of the number of $2$-chains in $\mathcal{L}$ and $\mathcal{R}$ which are asymmetric with respect to $M_{1}$. Let the function $\Phi:\mathcal{\mathcal{R}}\cup \mathcal{L}\rightarrow\{-1,+1\}$ be defined on the $2$-chains $\cT\in\mathcal{R}\cup \mathcal{L}$ such that:
	  \[
	  \Phi(\mathcal{T}) =
	  \begin{cases}
	  	+1 & $if $\cT$ is symmetric with respect to $M_{1},\\
	  	-1 & $otherwise.$
	  \end{cases}
	  \]

	  \noindent\textbf{Subcase 3.1:} $\mathcal{L}$ and $\mathcal{R}$ each have an even number (possibly zero) of asymmetric $2$-chains with respect to $M_{1}$.
	  
	  We claim that there exists a perfect matching such that its union with $M_{1}$ gives a Hamiltonian cycle of $\cCP_n$. Since the number of asymmetric $2$-chains in $\mathcal{R}$ is even, $\prod\limits_{{\cT}\in\mathcal{R}}\Phi({\cT})=+1$, and consequently, by appropriately concatenating paths as in Remark \ref{rem1}, there exists a path $R$ with endvertices $u_{1}$ and $u_{2n}$ whose vertex set is $\bigcup\limits_{i=1}^{2n}\{u_{i},v_{i}\}$ such that it contains all the edges in $M_{1}\cap (\bigcup\limits_{i=1}^{n}E({\cT}_{i}) )$. We remark that this path intersects exactly one edge of $\{xy\in E(G): x\in V({\cT}_{j}), y\in V({\cT}_{j+1})\}$, for each $j=1,\dots, n-1$.  By a similar reasoning, since $\prod\limits_{{T}\in\mathcal{L}}\Phi({\cT})=+1$, there exists a path $L$ with endvertices $u_{2n+1}$ and $u_{4n}$ whose vertex set is $\bigcup\limits_{i=2n+1}^{4n}\{u_{i},v_{i}\}$, such that it contains all the edges in $M_{1}\cap (\bigcup\limits_{i=n+1}^{2n}E({\cT}_{i}) )$. Once again, this path intersects exactly one edge of $\{xy\in E(G): x\in V({\cT}_{j}), y\in V({\cT}_{j+1})\}$, for each $j =n+1,\ldots,2n-1$.  These two paths, together with the edges $a$ and $d$ form the required Hamiltonian cycle of $\cCP_n$ containing $M_{1}$, proving our claim. We remark that this shows that there exists a perfect matching $M_{2}$ of $\cCP_n$ such that $M_{2}\cap \mathcal{X}=\{a,d\}$, $M_{1}\cap M_{2}=\emptyset$ and with $M_{1}\cup M_{2}$ inducing a Hamiltonian cycle of $\cCP_n$. One can similarly show that there exists a perfect matching $M_{2}'$ of $\cCP_n$ such that $M_{2}'\cap \mathcal{X}=\{b,c\}$, $M_{1}\cap M_{2}'=\emptyset$ and with $M_{1}\cup M_{2}'$ inducing a Hamiltonian cycle of $\cCP_n$.\\

	  \noindent\textbf{Subcase 3.2:} One of $\mathcal{L}$ and $\mathcal{R}$ has an odd number of asymmetric $2$-chains with respect to $M_{1}$.
	  
	  Without loss of generality, assume that $\mathcal{R}$ has an odd number of asymmetric $2$-chains with respect to $M_{1}$, that is, $\prod\limits_{\cT\in\mathcal{R}}\Phi(\cT)=-1$. Let $M_{2}$ be the perfect matching of $\cCP_n$ such that $|M_{2}\cap \mathcal{X}|=4$. We claim that $M_{1}\cup M_{2}$ induces a Hamiltonian cycle of $\cCP_n$. Since $\prod\limits_{{\cT}\in\mathcal{R}}\Phi({\cT})=-1$, by appropriately concatenating paths as in Remark \ref{rem2} we can deduce that $M_{1}\cup M_{2}$ contains the edge set of two disjoint paths $R_{1}$ and $R_{2}$, such that:
	  \begin{enumerate}
	  	\item $|V(R_{1})|=|V(R_{2})|=2n$;
	  	\item $V(R_{1})\cup V(R_{2})=\bigcup\limits_{i=1}^{2n}\{u_{i},v_{i}\}$;
	  	\item the endvertices of $R_{1}$ are $u_{1}$ and $v_{2n-1}$; and
	  	\item the endvertices of $R_{2}$ are $v_{2}$ and $u_{2n}$.
	  \end{enumerate}
	 We shall be using the fact that
	  $\{u_{1}u_{4n}, v_{2}v_{4n-1}, v_{2n-1}v_{2n+2}, u_{2n}u_{2n+1} \}=\{a,b,c,d\} = \mathcal{X} \subset M_2$.\\
	  
By Remark \ref{rem2} we can deduce that $M_{1}\cup M_{2}$ contains the edge set of two disjoint paths $L_{1}$ and $L_{2}$, such that:
\begin{enumerate}
	\item $|V(L_{1})|=|V(L_{2})|=2n$;
	\item $V(L_{1})\cup V(L_{2})=\bigcup\limits_{i=2n+1}^{4n}\{u_{i},v_{i}\}$;
	\item the endvertices of $L_{1}$ are $v_{2n+2}$ and $v_{4n-1}$; and 
	\item the endvertices of $L_{2}$ are $u_{2n+1}$ and $u_{4n}$.
\end{enumerate}
The concatenation of the following paths and edges gives a Hamiltonian cycle of $\cCP_n$ containing $M_{1}$:
\begin{linenomath}
	$$R_{1} v_{2n-1}v_{2n+2} L_{1} v_{4n-1}v_{2} R_{2} u_{2n}u_{2n+1} L_{2} u_{4n}u_{1}.$$
\end{linenomath}
This completes the proof.
\end{proof}

\section{Conclusions}
In this brief note, we have studied the $PMH$-property in two well-known families of cubic graphs, the Prisms and the Crossed Prisms. As far as the authors know, no cubic $PMH$ graphs with girth at least 8 have been found, and we believe that the technique of appropriately modifying the edge set of a graph to make it $PMH$ could be useful in this direction.

\end{document}